\documentclass[reqno]{amsart}
\usepackage{amsmath}
\usepackage[dvips]{graphicx}
\usepackage{amsfonts}
\usepackage{amssymb}
\usepackage{latexsym}
\graphicspath{{Img/}}
\newtheorem{theorem}{Theorem}
\theoremstyle{plain}

\newtheorem{corollary}{Corollary}

\newtheorem{definition}{Definition}
\newtheorem{example}{Example}

\newtheorem{problem}{Problem}
\newtheorem{proposition}{Proposition}
\newtheorem{remark}{Remark}

\numberwithin{equation}{section}

\begin{document}

\title[Explicit solutions for the Fermat problem for tetrahedra]{A class of Explicit solutions for the Fermat problem for tetrahedra}
\author{Anastasios N. Zachos}
\address{University of Patras, Department of Mathematics, GR-26500 Rion, Greece}
\email{azachos@gmail.com}

\begin{abstract}

We present a class of explicit solutions for the problem of minimization of the function $f(x,y,z)=\sum_{i=1}^{4}\sqrt{(x-x_{i})^2+(y-y_{i})^2+(z-z_{i})^2},$
which gives the location of the unique stationary (Fermat-Torricelli) point for four non-collinear and non-coplanar points $A_{i}=(x_{i},y_{i},z_{i}),$
determining tetrahedra, which are derived by a proper class of isosceles tetrahedra having four equal edges and two equal opposite edges. This class of explicit solutions contains Mehlhos and Glastier's explicit solutions (theoretical constructions) obtained in \cite{Mehlhos:00} and \cite{Glastier:93}, respectively.
\end{abstract}\maketitle

\section{Introduction}

In 1643, Fermat posed an optimization problem, which has not been studied by Ancient Greeks. The problem states as follows:
Given three points in the Euclidean plane, find a point having the minimum sum of distances from these three points.

The solution of the Fermat problem is called the Fermat point. There are two types of solutions for the Fermat problem in the Euclidean plane.

First type. The corner angles of the triangle formed by these three fixed points are less than $120$ degrees.
This type was discovered by Torricelli (a student of Gallileo) and it is called Torricelli solution.
The Torricelli solution refers to the Fermat-Torricelli point, which is strictly inside the triangle and its side of the triangle is seen by $120$ degrees.
This is the isogonal (equiangular) property of the Fermat-Torricelli point. The Fermat Torricelli point is the intersection point of three circles circumscribed around regular triangles located outside on each side of the triangle. Hence, the Fermat-Torricelli point is constructed for three non-collinear points in $\mathbb{R}^{2}$ using ruler and compass (Euclidean construction).

Second type. One of the corner angles of the triangle is greater or equal to $120$ degrees.
This type was discovered by Cavallieri and it is called Cavallieri solution.
The Cavallieri solution refers to the Fermat-Cavalieri point, which is the vertex of the obtuse corner angle of the triangle.

Extending the Fermat problem for four non-collinear points in $\mathbb{R}^{2},$ we consider the following two cases:

Case~1 If these four points form a convex quadrilateral, then the Fermat (Torricelli) point is the intersection of the two diagonals, This result was proved by Fagnano.

Case~2 If these four points form a non-convex quadrilateral, then the Fermat (Cavallieri) point is the vertex of the non-convex angle. This result is easily derived by using triangle inequalities.

The unsolvability of the Fermat problem by Euclidean constructions using ruler and compass for five points in $\mathbb{R}^{2}$ has been proved by Cockayne and Melzak in \cite{Bajaj:88} by applying Galois theory in a specific example $P_{1}P_{2}P_{3}P_{4}P_{5}$ for $P_{1}=(0,0),$ $P_{2}=(0,-1),$ $P_{3}=(0,1),$
$P_{4}=(a,b),$ $P_{5}=(a,-b).$ They derive an eighth degree polynomial equation with respect to the distance $P_{1}P_{min},$ where $P_{min}$ is the Fermat-Torricelli point. The coefficients of these polynomials depend on the integers $a, b.$ They observed that the eight degree polynomial equation contains a Galois group over the field of rationals, which does not have order $2^k,$ for $k$ a positive integer. Hence, this equation is not solvable by radicals. Therefore, $P_{1}P_{min}$ is not constructible using ruler and compass.
Furthermore, the unsolvability of Fermat problem by Euclidean constructions has been proved for five points in $\mathbb{R}^{2}$ using different examples by Bajaj in \cite{Bajaj:88} and by Mehlhos in \cite{Mehlhos:00}.





Let $A_{1}A_{2}A_{3}A_{4}$ be a tetrahedron and let $A_{i}=(x_{i},y_{i},z_{i}),$ $i=1,2,3,4.$ Then the Fermat's Problem states as follow:

\begin{problem}\label{WFN}
Find $(x,y,z)$ in $\mathbb{R}^{3},$ that minimizes:

\begin{equation}\label{objectivewfrn}
f(x,y,z)=\sum_{i=1}^{4}\sqrt{(x-x_{i})^2+(y-y_{i})^2+(z-z_{i})^2}.
\end{equation}

\end{problem}

\begin{definition}The solution to Problem~1 is called the Fermat point of a tetrahedron $A_{1}A_{2}A_{3}A_{4}$ and it is denoted by $A_{0}.$
\end{definition}

The two types of solutions for the Fermat Problem (Problem~1) are:

Case~1 $A_{0}$ is strictly inside of $A_{1}A_{2}A_{3}A_{4}.$ $A_{0}$ is called the Fermat-Torricelli point and it is a generalization of Torricelli's solution in $\mathbb{R}^{3}.$

Case~2 $A_{0}$ coincides with one of the vertices $\{A_{1},A_{2},A_{3},A_{4}\}.$ $A_{0}$ is called the Fermat-Cavallieri point and it is a generalization of Cavallieri'solution in $\mathbb{R}^{3}.$

It is worth mentioning that Synge in correspondence with Coxeter were the first who suggested to call a generalization of solution of Case~1 for a finite number of given points in $\mathbb{R}^{n}$ ($n\ge 2$) Torricelli-Fermat point or Fermat-Torricelli point (\cite[p.~2]{Synge:87}).

There are many references to the Fermat problem and called by different names. The Fermat-Weber problem refers to the Fermat problem for the location of industries and created a new science that is called "Location Science" (\cite{Wesolowsky:93}). The optimum location for the production of a good based on the fixed locations of the market and a finite number of raw material sources, which do not belong to the same line and determine the least-cost production location by computing the total costs of transporting raw materials from these sites to the production site and product from the production site to the market.

The Fermat-Steiner problem also refers to the tree solution for three non-collinear points to the plane, which was suggested by Courant and Robbins in \cite{CourantRobbins} and adopted by Gueron and Tessler in \cite{Gue/Tes:02} and has been applied by many researchers in the fields of Mathematical Chemistry and Mathematical Biology and especially in the geometric folding problem of proteins and aminoacids (\cite{SmithJangKim:07}).

We adopt Synge's Coxeter suggestion for the Fermat solution, which is the Fermat-Torricelli point for case~1 and use the Fermat-Cavalleri point for the case~2.

In 2000, Mehlhos gave an elegant proof for the unsolvability of the Fermat problem by considering for four non-coplanar and non-collinear points $A_{1}=(0,-1,0),$ $A_{2}=(0,1,0),$ $A_{3}=(1,0,0),$ $A_{4}=(1,0,1)$ in $\mathbb{R}^{3}$  without using Galois theory(\cite[Theorem~2,pp.~153-155]{Mehlhos:00}). He observed that the Fermat-Torricelli point $A_{0}=(x,0,z)$ belongs to the $x-z-$plane and by turning $\triangle A_{1}A_{0}A_{2}$ by $90^{\circ}$ in the $x-z-$plane $A_{0}$ is the intersection point of the two diagonals of the new derived convex quadrilateral $A_{1}^{\prime}A_{2}^{\prime}A_{3}A_{4}.$ The angle of rotation is called the twist angle and it is very useful for the computation of the Fermat-Torricelli point. Thus, by applying Fagnano's result $A_{0}$ is the Fermat-Torricelli point of $A_{1}^{\prime}A_{2}^{\prime}A_{3}A_{4}.$ By using as variable the angle $\varphi=\angle A_{3}O x$ and by turning the figure with the angle $\varphi$ such that $A_{1}^{\prime}A_{2}^{\prime}$ belong to the $z-$axis and by substituting $\sin\varphi=x$ and $\cos\varphi=\sqrt{1-x^2}$ to the system of equations of the two diagonals he derived the quartic equation $8x^4-4x^3-7x^2+2x+1=0,$ which gives $8x^3+4x^2-3x-1=0,$ for $x\neq 1.$ Thus, we are able to use Cardano-Ferrari's formulas, in order to find the solution of this cubic equation. Then, by substituting $x=t-\frac{1}{6},$ we get $t^3-\frac{11}{24}t-\frac{23}{432}=0$ or $432t^3-198 t-23==0.$ Taking into account the rational root theorem, this equation does not have rational roots, but three real roots (one positive and two negative), because the determinant $D=4 (-198)^3+27 (-23)^2=-31035285<0.$ Hence, the positive real solution consists of a cubic root, which does not give a rational number. Therefore, $\sin\varphi$ is not constructive, which yields that $A_{0}=(\frac{\cos^2\varphi}{1-\sin\varphi},0,\frac{\cos\varphi \sin\varphi}{1-\sin\varphi})$ cannot be constructed using ruler and compass. We consider Mehlhos approach as the first example that deals with an explicit solution of the Fermat Problem for a tetrahedron such that four edges of the tetrahedron are seen by four equal angles and the remaining two opposite (non-neighboring) edges are seen by two equal angles.

In 2014, Uteshev succeeded in finding an explicit analytical solution for the Fermat problem for a triangle in $\mathbb{R}^{2}$ and managed to express the coordinates of the Fermat-Torricelli point w.r to the coordinates of the three fixed points $A_{1}, A_{2}, A_{3}.$ The key idea is that he used duality in the Fermat problem and introduced a dual Fermat-Torricelli problem even for unequal positive numbers (weights), that correspond to the vertices of the triangle (\cite{Uteshev:}). The value of the objective function $f(x,y)=\sum_{i=1}^{4}b_{i}\sqrt{(x-x_{i})^2+(y-y_{i})^2}$ does not change if the weights $b_{1}, b_{2}, b_{3}$ become distances (side lengths of a dual triangle) and the distances $\sqrt{(x-x_{i})^2+(y-y_{i})^2}$ become weights, for $i=1,2,3.$ Unfortunately, the notion of duality does not seem to work for the Fermat problem for tetrahedra, because the Fermat-Torricelli problem is in general not constructible using ruler and compass, even for the unweighted case.

We consider six main attempts to find the Fermat-Torricelli point for tetrahedra in $\mathbb{R}^{3}:$

(i) Sturm's method of intersection of algebraic surfaces of order higher than two  (\cite{Sturm:13}) at the Fermat-Torricelli point. Egervary showed that the order of algebraic surfaces can be reduced to second order (\cite{Egervary:38}).

(ii) Weiszfeld introduced an algorithm, which gives a fixed point iteration method by solving the system of the three equations $\frac{\partial f}{\partial x }=0,$ $\frac{\partial f}{\partial y }=0,$ $\frac{\partial f}{\partial z }=0,$ of the objective function (\ref{objectivewfrn}) w.r to $x, y, z$ (\cite{Weis:37}, \cite{WeiszPlastria:09}). Kuhn proved the convergence of Weiszfeld algorithm in \cite{Kuhn:73}.

(iii) Synge's dynamic approach was given in \cite{Synge:87}. Synge (\cite{Synge:87}) was the first who used a dynamic process that contains an infinite number of steps (not a Euclidean construction) for the Fermat-Torricelli point using
"spindles" (spherical segments). He considered around the two opposite edges of the tetrahedron and created two isosceles triangles containing two angles $\pi -\angle A_{1}A_{0}A_{2}$  and $\pi -\angle A_{3}A_{0}A_{4}$ and by rotating two circular arcs having as chords the edges $A_{1}A_{2}$ $A_{3}A_{4},$ he showed that there is a common value such that $\angle A_{1}A_{0}A_{2}=\angle A_{1}A_{0}A_{2},$ which yields a unique touching point (Fermat-Torricelli point) of the two "spindles."

(iv) An $\epsilon$ approximation method of ovals ($m-$ellipsoids) back to Descartes.
An ellipsoid is a triaxial quadratic surface, which is given in Cartesian coordinates by the equation $(\frac{x}{a})^2+(\frac{y}{b})^2+(\frac{z}{c})^2=1$
Straud generalized the"thread" construction for an ellipsoid analogous to the taut pencil and string construction of the ellipse (\cite[pp.~19-22]{HilbertVossen:99}), with respect to two fixed confocal points. In a private letter (1638) Descartes invited Fermat to investigate properties of ellipses with four confocal points. These "ellipses" are called 'egglipses' or Descartes ovals or multiconics (\cite{GroStrempel:98}). The concentration of tetrafocal ellipsoids to a specific value for the objective functions $f(x,y,z)$ lead to the unique Fermat-Torricelli point. In general, the Fermat problem is an NP hard problem to compute the Fermat-Torricelli point. Chandrasekaran and Tamir invented an ellipsoid method, which leads to a polynomial method to construct an approximate solution for the Fermat problem of fixed accuracy in $\mathbb{R}^{n}.$ Thus, the Fermat problem may be solved in polynomial time (\cite{ChandrasekaranTamir}).

(v) Explicit solutions of Glastier, Kupitz-Martini, Mehlhos.
Let $A_{1}=(0,-1,0),$ $A_{2}=(0,1,0),$ $A_{3}=(1,0,0),$ $A_{4}=(1,0,1)$ be a tetrahedron in $\mathbb{R}^{3}.$

First type of explicit solution (Regular tetrahedra)
In 1993, Glastier proved that the centroid $G$ and the Fermat-Torricelli point $A_{0}$ of a regular tetrahedron $A_{1}A_{2}A_{3}A_{4}$ coincide, because the following balancing condition of unit vectors hold (\cite{Glastier:93}):
\[ \frac{\vec{X A_{1}}}{|X A_{1}|}+\frac{\vec{X A_{2}}}{|X A_{2}|}+\frac{\vec{X A_{3}}}{|X A_{3}|}+\frac{\vec{X A_{4}}}{|X A_{4}|}=0, \]
for $X=G$ and $X=A_{0}.$ By squaring both sides of the equation, the equiangular property of the Fermat-Torricelli point for regular tetrahedra is obtained:
\[\angle A_{i}A_{0}A_{j}=\cos^{-1}(-1/3)\approx 109^{\circ}28^{\prime},\] for $i,j=1,2,3,4, i\neq j.$
Furthermore, he verified a result well-known in chemistry regarding the molecular structure of methane $CH_{4},$ by observing that the six bond angles HCH in the methane molecule $CH_{4}$ are equal to $\approx 109^{\circ}28^{\prime}$ and by applying the Fermat solution for a regular tetrahedron (\cite{Glastier:93}).The four positions of hydrogen $H$ yield a regular tetrahedron and the carbon $C$ is located at the corresponding Fermat-Torricelli point.

Second type of explicit solution (Isosceles tetrahedra)
A tetrahedron is called isosceles if all its two faces (triangles) are congruent. The centroid (barycenter) of an isosceles tetrahedron coincides with the Fermat-Torricelli point and with the center of the circumscribed sphere.  Various characterizations for isosceles tetrahedra are given by Bogatyi in Arnold's Problems (\cite[pp.~188-192]{Arnold:04}) and by Kupitz-Martini (\cite{KupitzMartini:94}).

Third type of explicit solution (Glastier)
Let $OA_{1}A_{2}A_{3}$ be a tetrahedron, such that $O=(0,0,0),$ $A_{1}=(1,0,0),$ $A_{2}=(0,1,0),$ $A_{3}=(0,0,1).$ The Fermat-Torricelli point can be assumed to be placed at $A_{0}=(x,x,x).$ Thus, the objective function $f(x,y,z)$ takes the form $f(x,x,x)=\sqrt{3} x + 3\sqrt{2x^2+(x-1)^2}.$
By solving $\frac{\partial f}{\partial x}=0,$ we get $12x^2-8x+1=0,$ which gives $x=\frac{1}{6}.$ We note that the Fermat-Torricelli point $A_{0}=(\frac{1}{6},\frac{1}{6},\frac{1}{6})$ possesses the equiangular property $\angle A_{i}A_{0}A_{j}=\arccos (-\frac{1}{3}),$ for $i,j=1,2,3,4,$ $i\neq j.$ An alternative proof is to consider a point $A_{5}=(-\frac{1}{3},-\frac{1}{3},-\frac{1}{3})$ that lies on the ray formed by $A_{0}O.$ Hence, $A_{5}A_{1}A_{2}A_{3}$ is a regular tetrahedron, which satisfies the equation
\[ \frac{\vec{X A_{1}}}{|X A_{1}|}+\frac{\vec{X A_{2}}}{|X A_{2}|}+\frac{\vec{X A_{3}}}{|X A_{3}|}+\frac{\vec{X A_{5}}}{|X A_{5}|}=0 \]
or
\[\frac{\vec{X A_{1}}}{|X A_{1}|}+\frac{\vec{X A_{2}}}{|X A_{2}|}+\frac{\vec{X A_{3}}}{|X A_{3}|}+\frac{\vec{X O}}{|X O|}=0.\]
Hence, we get $X=A_{0}.$

Fourth type of explicit solution (Mehlhos)
Let $A_{1}=(0,-1,0),$ $A_{2}=(0,1,0),$ $A_{3}=(1,0,0),$ $A_{4}=(1,0,1)$ be a tetrahedron in $\mathbb{R}^{3}.$ By following the same process that was used by Mehlos to show the unsolvability of the Fermat problem for tetrahedra with ruler and compass, Mehlhos derived the following explicit solution

\[A_{0}=(\frac{\cos^2\varphi}{1-\sin\varphi},0,\frac{\cos\varphi \sin\varphi}{1-\sin\varphi}),\]
where $\varphi=\angle A_{3}O x.$ This is the first explicit solution for the Fermat-Torricelli problem, such that the equiangular property of the Fermat-Torricelli point does not hold for an non-regular tetrahedron. The non-isogonal property of the Fermat-Torricelli point for tetrahedra has been possibly known to Sturm and Lindelof (\cite{Sturm:13},\cite{Sturm:84},\cite{Lindelof:67}) showed by Synge in \cite{Synge:87} and proved by Abu-Abas, Abu-Saymeh and Hajja for higher dimensional simplexes in \cite{AbuAbbasHajja} and \cite{abuSaymehHajja:97}).

(vi) Characterizations of the solutions for the Fermat problem for tetrahedra.
Sturm, Kupitz-Martini, Eriksson, Noda, Sakai, Morimoto derived various characterizations for the Fermat-Torricelli point for tetrahedra in $\mathbb{R}^{3}$ (\cite{Sturm:13}, \cite{KupitzMartini:94}\cite{Eriksson:97}, \cite{NodaSakaiMorimoto:91}).

The existence of the Fermat-Torricelli point $A_{0}$ in $\mathbb{R}^{3}$  is easily derived by a well known theorem of Weirstrass (\cite[Corollary, p~111]{Tikhomirov:90}):

If the function $f(x,y,z)$ is continuous for every $(x,y,z)\in \mathbb{R}^{3}$ and $\lim f(x,y,z)=\infty,$ for $x^2+y^2+z^2\to \infty,$ then the unconstrained problem is solvable. Therefore, the objective function $f(x,y,z)$ is solvable and a solution (Fermat point) exists.
The uniqueness of the Fermat point is obtained by the convexity of the Euclidean norm (distance).
A complete proof for the convexity of distances in hyperbolic spaces is given by Thurston in his classical book \cite[Theorem~2.5.8.pp.~90-94]{Thurston:97}. The proof is the same in $\mathbb{R}^{n}$ (\cite[Exercise~2.5.13 Case(b).p.~94]{Thurston:97}).

Sturm and Lindelof were the first, who gave a complete characterization of the
solutions of the Fermat-Torricelli problem for $m$ given
points in $\mathbb{R}^{n}$ (\cite{Sturm:84},\cite{Lindelof:67}).

We need the following three results, which deal with the uniqueness and the characterization of solutions in the three dimensional case (\cite[Theorem~2.5.8.pp.~90-94]{Thurston:97}, \cite[Theorem~18.3, pp.~237]{BolMa/So:99},\cite[Theorem,p.~863]{Spira:71}, \cite[Property~3, p.~154]{Mehlhos:00} or \cite[Formulas (6,2), (6.3),p.~8]{Synge:87}).

Let $A,B$ two points in $\mathbb{R}^{3}.$

\begin{theorem}(Thurston's strict convexity of Euclidean distance functions)\label{thm1}
The distance function $d(A,B)$ considered as a map $d:\mathbb{R}^{3}\times \mathbb{R}^{3} \to \mathbb{R}$, is convex. The composition $d\circ \gamma$ is strictly convex for any line $\gamma$ in $\mathbb{R}^{3}\times \mathbb{R}^{3},$ whose projections to the two factors ane distinct lines.
\end{theorem}

We denote by $\vec {u}(j,i)\equiv \frac{\overrightarrow{A_{j}A_{i}}}{\|A_{j}A_{i}\|}$ the unit vector from $A_{j}$ to $A_{i}$ for $i,j=0,1,2,3,4.$


\begin{theorem}\label{thm2}
(I) If  $ \|\sum_{j=1,j\ne i}^{4}\vec {u}(j,i)\|>1,$
for each $i=1,2,3,4,$ then

(a) $A_{0}$ is strictly inside of the tetrahedron $A_{1}A_{2}A_{3}A_{4},$

(b) $ \sum_{i=1}^{4}\vec {u}(0,i)=\vec{0}$

(Fermat-Torricelli solution).

(II) If  $ \|{\sum_{j=1,j\ne i}^{4}\vec {u}(j,i)}\|\le 1$ for some $i=1,2,3,4,$ then  $A_{0}= A_{i}.$

(Fermat-Cavallieri solution).

\end{theorem}

\begin{theorem}(Spira-Synge's relations)\label{thm3}
The following relations hold:
\begin{equation}\label{coseq1}
\cos\angle A_{1}A_{0}A_{2}=\cos\angle A_{3}A_{0}A_{4},
\end{equation}

\begin{equation}\label{coseq2}
\cos\angle A_{2}A_{0}A_{3}=\cos\angle A_{1}A_{0}A_{4},
\end{equation}

\begin{equation}\label{coseq3}
\cos\angle A_{1}A_{0}A_{3}=\cos\angle A_{2}A_{0}A_{4}
\end{equation}

and

\begin{equation}\label{coseq4}
1+\cos\angle A_{1}A_{0}A_{2}+\cos\angle A_{1}A_{0}A_{3}+\cos\angle A_{2}A_{0}A_{4}=0.
\end{equation}

\end{theorem}


In this paper, we focus on the fifth approach and we introduce a class of explicit solutions for the Fermat problem for tetrahedra in $\mathbb{R}^{3},$
which contains the explicit solutions given by Mehlhos and Glastier.



\section{Explicit solution for the Fermat problem for isosceles tetrahedra and almost platonic tetrahedra}

We start by giving three definitions for three types of tetrahedra.

\begin{definition}
A Platonic tetrahedron is a regular polyhedron in which all edges are equal.
\end{definition}

\begin{definition}
An isosceles tetrahedron is a non-regular polyhedron in which each pair of opposite edges are equal.
\end{definition}

\begin{definition}
 An almost Platonic tetrahedron is an isosceles tetrahedron having four equal edges and a pair of equal opposite edges.
\end{definition}

Let $A_{1}A_{2}A_{3}A_{4}$ be an isosceles tetrahedron and let $A_{0}$ be the Fermat-Torricelli point. We denote by $S (O;R)$ the circumscribed sphere with center $O$ and radius $R,$ where $A_{i}\in S (O;R)$ for $i=1,2,3,4.$

We denote by $P^{i}_{jk}$ the orthogonal projection of $A_{i}$ to the plane defined by $\triangle A_{j}OA_{k}.$

We denote by $\angle (ij) \equiv \angle A_{i}O A_{j}.$

We denote by $\angle (i,jk) \equiv \angle A_{i}OP^{i}_{jk},$ for $i,j,k=1,2,3,4.$

We denote by $\omega^{i}_{jk}=\angle P^{i}_{jk}O A_{1}.$

We set $A_{1}O=A_{2}O=A_{3}O=A_{4}O=R.$
\begin{proposition}\label{prop1}
If $A_{1}A_{2}=A_{3}A_{4}=a,$ $A_{1}A_{3}=A_{2}A_{4}=b,$ $A_{1}A_{4}=A_{2}A_{3}=c,$ then $A_{0}=O.$
\end{proposition}

\begin{proof}
By using the cosine law in $\triangle A_{1}OA_{2}$ and $\triangle A_{3}OA_{4},$ we get

\begin{equation}\label{eqqq1}
a^{2}=2R^2-2R^2\cos\angle (12)
\end{equation}

and

\begin{equation}\label{eqqq2}
a^{2}=2R^2-2R^2\cos\angle (34).
\end{equation}

By subtracting (\ref{eqqq2}) from (\ref{eqqq1}), we get $\angle (12)=\angle (34).$

By using the cosine law in $\triangle A_{1}OA_{3}$ and $\triangle A_{2}OA_{4},$ we get

\begin{equation}\label{eqqq3}
b^{2}=2R^2-2R^2\cos\angle (13)
\end{equation}

and

\begin{equation}\label{eqqq4}
b^{2}=2R^2-2R^2\cos\angle (24).
\end{equation}

By subtracting (\ref{eqqq3}) from (\ref{eqqq4}), we get $\angle (13)=\angle (24).$

By using the cosine law in $\triangle A_{2}OA_{3}$ and $\triangle A_{1}OA_{4},$ we get

\begin{equation}\label{eqqq5}
c^{2}=2R^2-2R^2\cos\angle (23)
\end{equation}

and

\begin{equation}\label{eqqq6}
c^{2}=2R^2-2R^2\cos\angle (14).
\end{equation}

By subtracting (\ref{eqqq5}) from (\ref{eqqq6}), we get $\angle (23)=\angle (14).$

Without loss of generality, we express the unit vectors $\vec {u}(0,i)$ for $i=1,2,3,4,$ in the following form:

\begin{equation}\label{spherical1}
\vec {u}(0,1)=(1,0,0),
\end{equation}

\begin{equation}\label{spherical2}
\vec {u}(0,2)=(\cos\angle (12),\sin\angle (12),0),
\end{equation}

\begin{equation}\label{spherical3}
\vec {u}(0,3)=(\cos \angle (3,12) \cos\omega^{3}_{12},\cos \angle (3,12) \sin\omega^{3}_{12},\sin \angle (3,12) ),
\end{equation}

\begin{equation}\label{spherical4}
\vec {u}(0,4)=(\cos \angle (4,12) \cos\omega^{4}_{12},\cos \angle (4,12) \sin\omega^{4}_{12},\sin \angle (4,12) ).
\end{equation}

By taking the inner products $\vec {u}(0,1)\cdot \vec {u}(0,3),$ $\vec {u}(0,1)\cdot \vec {u}(0,4),$ $\vec {u}(0,2)\cdot \vec {u}(0,3),$ $\vec {u}(0,3)\cdot \vec {u}(0,3),$ we get:
\begin{equation}\label{trginometric13}
\cos \angle (3,12) \cos\omega^{3}_{12}=\cos\angle (13),
\end{equation}

\begin{equation}\label{trginometric14}
\cos \angle (4,12) \cos\omega^{4}_{12}=\cos\angle (14),
\end{equation}


\begin{equation}\label{trginometric23}
\cos\angle (12)\cos \angle (3,12) \cos\omega^{3}_{12}+\sin\angle (12)\cos \angle (3,12) \sin\omega^{3}_{12}=\cos\angle (23),
\end{equation}

\begin{equation}\label{trginometric24}
\cos\angle (12)\cos \angle (4,12) \cos\omega^{4}_{12}+\sin\angle (12)\cos \angle (4,12) \sin\omega^{4}_{12}=\cos\angle (24),
\end{equation}

\begin{eqnarray}\label{trginometric34}
\cos \angle (3,12) \cos\omega^{3}_{12}\cos \angle (4,12) \cos\omega^{4}_{12}+\nonumber\\
\cos \angle (3,12) \sin\omega^{3}_{12}\cos \angle (4,12) \sin\omega^{4}_{12}+\sin \angle (3,12)\sin \angle (4,12)= \cos \angle (34).
\end{eqnarray}

By substituting (\ref{trginometric13}) in (\ref{trginometric23}), solving w.r. to $\cos \angle (3,12) \sin\omega^{3}_{12}$ and by squaring both parts of the derived equation and (\ref{trginometric13}), respectively and by adding these two equations, we eliminate $\omega^{3}_{12}$ and by setting $u=\cos\angle (12),$ $v=\cos\angle (23),$ $w=\cos\angle (13)$ in the derived equation, we obtain:

\begin{equation}\label{312cos}
\cos^{2} (3,12)=\frac{v^2+w^2-2 u v w}{1-u^2},
\end{equation}

Similarly, by substituting (\ref{trginometric14}) in (\ref{trginometric24}), solving w.r. to $\cos \angle (4,12) \sin\omega^{4}_{12}$ and by squaring both parts of the derived equation and (\ref{trginometric14}), respectively and by adding these two equations we eliminate $\omega^{4}_{12}$ and by setting $u=\cos\angle (12),$ $v=\cos\angle (23),$ $w=\cos\angle (13)$ in the derived equation, we obtain:

\begin{equation}\label{412cos}
\cos^{2} (4,12)=\frac{u^2+v^2-2 u v w}{1-u^2}.
\end{equation}

By setting $u=\cos\angle (12),$ $v=\cos\angle (23),$ $w=\cos\angle (13)$ and by substituting (\ref{312cos}), (\ref{412cos}), (\ref{trginometric13}), (\ref{trginometric13}) in (\ref{trginometric34}) and taking into account that $\angle (3,12)=-\angle (4,12),$ we get:

\begin{equation}\label{uvw}
u= v w +\frac{(u w - v)(u v - w)}{1-u^2} - \frac{1-u^2-v^2-w^2+2 u v w}{1-u^2}.
\end{equation}

Hence, by solving w.r to u, we obtain that:

\[ u=-1-v-w \]

or

\[1+\cos\angle (12)+\cos\angle (13)+\cos\angle (24)=0,\]
which gives

\[\vec {u}(0,1)+\vec {u}(0,2)+\vec {u}(0,3)+\vec {u}(0,4)=0.\]

Therefore $O=A_{0},$ because the Fermat-Torricelli point $A_{0}$ is unique.

\end{proof}

\begin{corollary}\label{cor2}
If $A_{1}A_{2}A_{3}A_{4}$ is an almost platonic tetrahedron, then $A_{0}=O.$
\end{corollary}

\begin{proof}
It is a direct consequence of Proposition~\ref{prop1} for $a=b$ or $a=c$ or $b=c.$
\end{proof}

\begin{corollary}\label{cor3}
If $A_{1}A_{2}A_{3}A_{4}$ is a platonic tetrahedron, then $A_{0}=O.$
\end{corollary}

\begin{proof}
It is a direct consequence of Proposition~\ref{prop1} for $a=b=c.$
\end{proof}


\begin{definition}
The natural Fermat-Torricelli point in $\mathbb{R}^{3}$ (Fig~1) is the Fermat-Torricelli point of an almost Platonic tetrahedron with four equal edges $a\sqrt{3}$ and two opposite equal edges $a\sqrt{2},$ where $a=R$ (radius of circumscribed sphere) is a positive real number.
\end{definition}

\begin{figure}\label{figg1}
\centering
\includegraphics[scale=0.70]{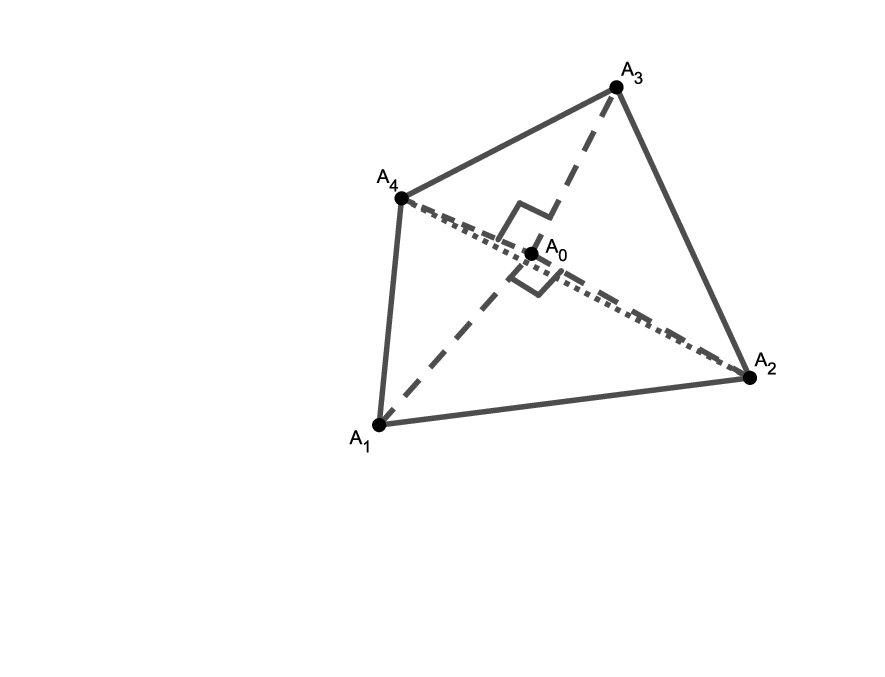}
\caption{The natural Fermat-Torricelli point in $\mathbb{R}^{3}:$ $\angle A_{1}A_{0}A_{2}=\angle A_{3}A_{0}A_{4}=90^{\circ}$ and $\angle A_{1}A_{0}A_{3}=\angle A_{2}A_{0}A_{4}=\angle A_{1}A_{0}A_{4}=\angle A_{2}A_{0}A_{3}=120^{\circ}.$ }
\end{figure}

\begin{proposition}\label{anglenaturalft}
If $A_{0}$ is the natural Fermat-Torricelli point in $\mathbb{R}^{3}$ of an almost Platonic tetrahedron $A_{1}A_{2}A_{3}A_{4},$ such that $A_{1}A_{2}=A_{3}A_{4}=a\sqrt{2},$ then $\angle A_{1}A_{0}A_{2}=\angle A_{3}A_{0}A_{4}=90^{\circ}$ and $\angle A_{1}A_{0}A_{3}=\angle A_{2}A_{0}A_{4}=\angle A_{1}A_{0}A_{4}=\angle A_{2}A_{0}A_{3}=120^{\circ}.$
\end{proposition}

\begin{proof}
By applying the cosine law in $\triangle A_{1}A_{0}A_{2},$ $\triangle A_{1}A_{0}A_{2},$ we get:

\[\cos\angle A_{1}A_{0}A_{2}=\cos \angle A_{2}A_{0}A_{4}=0\]
which implies that:
\begin{equation}\label{equal1}
\angle A_{1}A_{0}A_{2}=\angle A_{3}A_{0}A_{4}=90^{\circ}.
\end{equation}
By substituting (\ref{equal1}) in (\ref{coseq4}), we derive that:

\[\cos\angle A_{1}A_{0}A_{3}=\cos\angle A_{2}A_{0}A_{3}=-\frac{1}{2},\]
which yields
\[\angle A_{1}A_{0}A_{3}=\angle A_{2}A_{0}A_{4}=\angle A_{1}A_{0}A_{4}=\angle A_{2}A_{0}A_{3}=120^{\circ}.\]

\end{proof}

\begin{remark}
The natural Fermat-Torricelli point  coincides with the Fermat-Torricelli point in $\mathbb{R}^{2}$ for three non-collinear points $\{A_{1}A_{2}A_{3}\},$ which form a triangle with angles $\angle A_{i}A_{0}A_{j}<120^{\circ}-\epsilon,$ for $\epsilon$ a positive real number and $i,j=1,2,3$ (fig~2). We note that the two right angles that appeared in the three dimensional case vanish, but the other angles $\angle A_{1}A_{0}A_{2}=\angle A_{2}A_{0}A_{3}=\angle A_{1}A_{0}A_{3}=120^{\circ}$ angles remain (Theorem of Torricelli) and the equiangular property holds in the two dimensional case.
\end{remark}

\begin{figure}\label{figg2}
\centering
\includegraphics[scale=0.70]{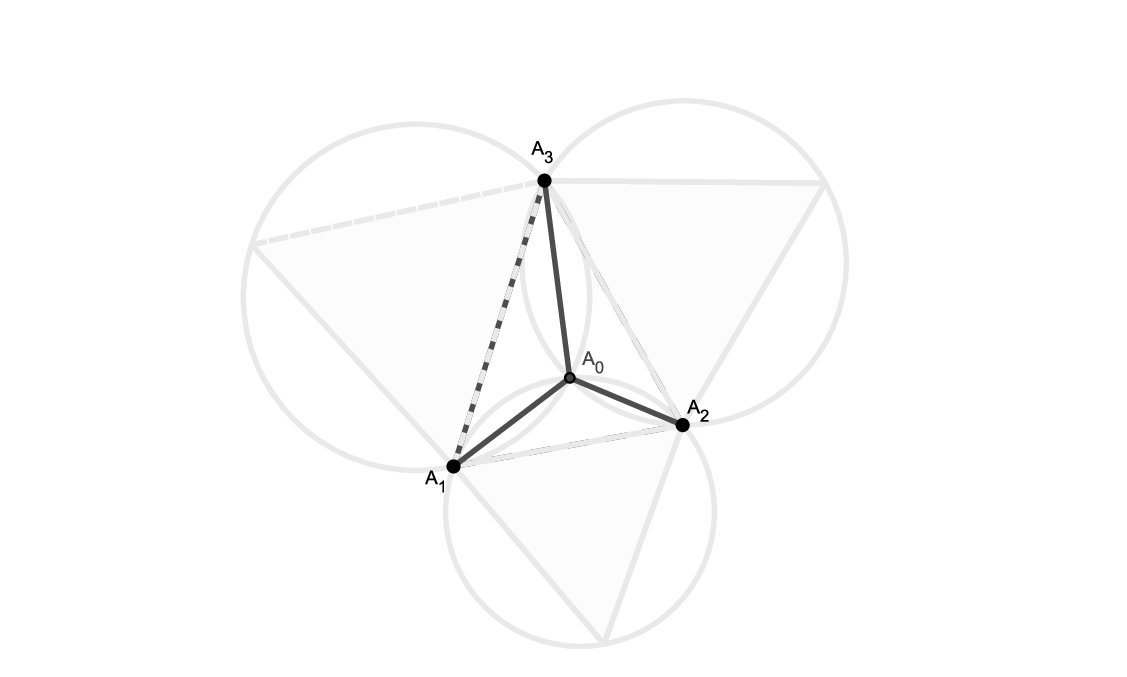}
\caption{The natural Fermat-Torricelli point coincides with the Fermat-Torricelli point in $\mathbb{R}^{2}$  $\angle A_{1}A_{0}A_{2}=\angle A_{2}A_{0}A_{3}=\angle A_{1}A_{0}A_{3}=120^{\circ}.$ }
\end{figure}

We may expand or contract isosceles tetrahedra, almost platonic tetrahedra and regular tetrahedra to non regular tetrahedra with respect to their Fermat-Torricelli point $A_{0}$ whose Fermat-Torricelli point $A_{0}^{\prime}$ remains the same $A_{0}=A_{0}^{\prime}.$
Let $A_{1}A_{2}A_{3}A_{4}$ be an isosceles tetrahedron or an almost platonic tetrahedron or a regular tetrahedron and $A_{i}^{\prime}$ be a point that lies on the ray $A_{i}A_{0}$ for $i=1,2,3,4,$ such that $A_{i}^{\prime}\in [A_{0},A_{i}]$ or $A_{i}\in [A_{0},A_{i}^{\prime}].$

\begin{proposition}\label{nonregulartetrahedraft}
If  $ \|\sum_{j=1,j\ne i}^{4} \frac{\vec{A_{j}^{\prime}A_{i}^{\prime}}}{|A_{j}^{\prime}A_{i}^{\prime}|}\|>1,$ for each $i=1,2,3,4,$ then $A_{0}^{\prime}=A_{0}.$
\end{proposition}

\begin{proof}
By applying Theorem~\ref{thm2}, we obtain the balancing condition of unit vectors with respect to the Fermat-Torricelli point $A_{0}^{\prime}$

\[\sum_{i}^{4}\frac{\vec{A_{0}^{\prime}A_{i}^{\prime}}}{|A_{0}^{\prime}A_{i}^{\prime}|}=0,\]

or

\[\sum_{i}^{4}\frac{\vec{A_{0}^{\prime}A_{i}}}{|A_{0}^{\prime}A_{i}|}=0,\]
which gives $A_{0}^{\prime}=A_{0}.$

\end{proof}

\section{A class of explicit solutions using quadratic equations}

We introduce a class of explicit solutions for the Fermat problem for tetrahedra derived by almost platonic tetrahedra

Let $A_{1}A_{2}A_{3}A_{4}$ be a tetrahedron and $A_{0}$ be the Fermat-Torricelli point inside $A_{1}A_{2}A_{3}A_{4}.$

We denote by $a_{ij}$ the distance $|A_{i}A_{j}|,$ for $i,j=0,1,2,3,4.$

We set $x=a_{01},$ $y=a_{02},$ $z=a_{03},$ $d=a_{04},$ $\omega=\angle A_{1}A_{0}A_{2},$ $\phi=\angle A_{2}A_{0}A_{3}.$

Let $A_{1}^{\prime}A_{2}^{\prime}A_{3}^{\prime}A_{4}^{\prime}$ be an almost platonic tetrahedron, such that each vertex $A_{i}^{\prime}$ lies on the ray $A_{i}A_{0}$ or $A_{i}^{\prime}\in [A_{0},A_{i}]$ for $i=1,2,3,4.$
Without loss of generality, we may assume that $\omega=\angle A_{1}A_{0}A_{2}=\angle A_{3}A_{0}A_{4}$ and $\phi=\angle A_{1}A_{0}A_{3}=\angle A_{2}A_{0}A_{3}=\angle A_{1}A_{0}A_{4}=\angle A_{2}A_{0}A_{4}.$

Given $a_{12}$ and $\omega,$ we choose the values of $a_{13}$ and $a_{34},$ such that the following inequalities are satisfied:

\[
\frac{\sqrt{2}}{2}\frac{a_{12}}{1-\cos\omega} \sin\arccos (\frac{-1-\cos\omega}{2}) <a_{13}< \frac{\sqrt{2}}{2}\frac{a_{12}}{1-\cos\omega}\]

and

\[z\sin\omega < a_{34} <z \sin\omega\]
where $z$ is a positive root of the quadratic equation \[z^2 -2 \frac{\sqrt{2}}{2}\frac{a_{12}}{1-\cos\omega}\cos\phi z +(\frac{\sqrt{2}}{2}\frac{a_{12}}{1-\cos\omega})^2-a_{13}^2=0.\]

\begin{theorem}\label{theor2}
If $a_{23}=a_{13},$  then

\begin{equation}\label{eq10fund1}
x(a_{12},\omega)=y(a_{12},\omega)=\frac{\sqrt{2} a_{12}}{2 (1-\cos\omega)},
\end{equation}

\begin{equation}\label{eq14fund2}
z(\phi_,x(a_{12},\omega),a_{13})=x(a_{12},\omega)\cos \phi+\sqrt{a_{13}^2- x^2(a_{12},\omega)\sin^{2}\phi}
\end{equation}

\begin{equation}\label{eq17fund3}
d(z(\phi_,x(a_{12},\omega),a_{13}),\omega,a_{34})=z(\phi_,x(a_{12},\omega),a_{13})\cos\omega+\sqrt{a_{34}^2-z^2(\phi_,x(a_{12},\omega),a_{13})\sin^2\omega}
\end{equation}

\begin{eqnarray}\label{eq18fund4}
a_{14}=a_{24}=\nonumber\\ \sqrt{x^2 (a_{12},\omega) + d^2(z(\phi_,x(a_{12},\omega),a_{13}),\omega,a_{34})- 2x(a_{12},\omega) d(z(\phi_,x(a_{12},\omega),a_{13}),\omega,a_{34})\cos\phi}.
\end{eqnarray}

\end{theorem}

\begin{proof}

By substituting $\omega=\angle A_{1}A_{0}A_{2}=\angle A_{3}A_{0}A_{4}$ and $\phi=\angle A_{1}A_{0}A_{3}=\angle A_{2}A_{0}A_{3}=\angle A_{1}A_{0}A_{4}=\angle A_{2}A_{0}A_{4},$ in (\ref{coseq4}), we obtain:

\begin{equation}\label{eq7}
\cos\phi=\frac{-1-\cos\omega}{2}<0.
\end{equation}

By applying the cosine law in $\triangle A_{2}A_{0}A_{3},$ $\triangle A_{1}A_{0}A_{4},$ $\triangle A_{1}A_{0}A_{3},$
$\triangle A_{2}A_{0}A_{4},$ $\triangle A_{1}A_{0}A_{2},$ $\triangle A_{3}A_{0}A_{4},$ we get:
\begin{equation}\label{eq1}
\cos\phi=\frac{y^2+z^2-a_{23}^2}{2yz},
\end{equation}

\begin{equation}\label{eq2}
\cos\phi=\frac{x^2+d^2-a_{14}^2}{2xd}
\end{equation}

\begin{equation}\label{eq3}
\cos\phi=\frac{x^2+z^2-a_{13}^2}{2xz}
\end{equation}

\begin{equation}\label{eq4}
\cos\phi=\frac{y^2+d^2-a_{24}^2}{2yd}
\end{equation}

\begin{equation}\label{eq5}
\cos\omega=\frac{x^2+y^2-a_{12}^2}{2xy}
\end{equation}

\begin{equation}\label{eq6}
\cos\omega=\frac{z^2+d^2-a_{34}^2}{2zd}
\end{equation}

By subtracting (\ref{eq1}) from (\ref{eq3}) and taking into account that $a_{31}=a_{23},$ we derive that $x=y$ or $0<y x=z^2-a_{23}^2,$ which gives
$x=y$ or $z>a_{23}.$ The inequality $z>a_{23}$ is impossible, because (\ref{eq1}) yields $z^2<y^2+z^2<a_{23}^2$ ($\phi$ is an obtuse angle), which implies that $z<a_{23}.$ Therefore, we obtain $x=y.$
By substituting $x=y$ in (\ref{eq5}), we obtain (\ref{eq10fund1}).

By substituting (\ref{eq10fund1}) in (\ref{eq3}), we get:

\begin{equation}\label{eq11}
z^2 -2 x(a_{12},\omega)\cos\phi z +x^{2}(a_{12},\omega)-a_{13}^2=0.
\end{equation}

The lower and upper bound of $a_{13}$
\[\frac{\sqrt{2}}{2}\frac{a_{12}}{1-\cos\omega} \sin\arccos (\frac{-1-\cos\omega}{2}) <a_{13}< \frac{\sqrt{2}}{2}\frac{a_{12}}{1-\cos\omega}\]
yields the positive root (\ref{eq14fund2})

By substituting (\ref{eq14fund2}) in (\ref{eq6}), we get:

\begin{equation}\label{eq15}
d^2-2 z(\phi_,x(a_{12},\omega),a_{13})\cos\omega d+ z^2(\phi_,x(a_{12},\omega),a_{13})-a_{34}^2=0
\end{equation}

The lower and upper bound of $a_{34}$

\[z \sin\omega <a_{34}<z\]

yields the positive root (\ref{eq17fund3}) of the quadratic equation (\ref{eq15}).

By substituting (\ref{eq10fund1}), (\ref{eq17fund3}) in (\ref{eq2}), (\ref{eq4}), we derive (\ref{eq18fund4}).

\end{proof}

\begin{example}
Let $\omega=105^{\circ},$  $a_{12}=3,$ $a_{23}=a_{13}=4,$ $a_{34}=4.5$ By substituting these values in (\ref{eq7}), (\ref{eq10fund1}), (\ref{eq14fund2}), (\ref{eq17fund3}), (\ref{eq18fund4}), we get:

$\phi=111.752^{\circ},$ $a_{14}=a_{24}=3.96628,$    $x=y=1.89071,$ $z=2.89323$ and $d=2.85566.$
Therefore, we obtain a tetrahedron, which is derived by an almost platonic tetrahedron with two equal angles $105^{\circ}$ and $111.752^{\circ}$ having each vertex at a prescribed ray, which is formed by the Fermat-Torricelli point $A_{0}$ and a vertex of the almost platonic tetrahedron.

\end{example}

\begin{example} Derivation of Mehlhos explicit solution in \cite{Mehlhos:00}.
Let $\omega=106.654^{\circ},$  $a_{12}=2,$ $a_{23}=a_{13}=\sqrt{2},$ $a_{34}=1.$ By substituting these values in (\ref{eq7}), (\ref{eq10fund1}), (\ref{eq14fund2}), (\ref{eq17fund3}), (\ref{eq18fund4}), we get:

$\phi=110.898^{\circ},$ $a_{14}=a_{24}=\sqrt{3},$    $x=y=1.24679,$ $z=0.35732$ and $d=0.837174.$

\end{example}

\section{Concluding Remarks}
Various properties of closed polyhedra whose vertices exist on prescribed rays that meet at a point have been given by A.D. Alexandrov in \cite[Chapter~9]{Alexandrov}. By enriching this intersection point with the minimum property of Fermat's problem for tetrahedra, we are able to prescribe tetrahedra derived by almost platonic tetrahedra having each vertex at a prescribed ray, such that the corresponding Fermat-Torricelli point remains the same. Theorem~\ref{theor2} may offer a class of explicit solutions to the Fermat problem for tetrahedra, which partially answers a question posed by Ivanov-Tuzhilin's preprint for open problems in \cite{IvanovTuzhilin:} that deals with the derivation of theoretical constructions for Fermat points for four points in $\mathbb{R}^{3}.$


\begin{thebibliography}{99}
\bibitem{WeiszPlastria:09} E. Weiszfeld, F. Plastria, \emph{On the point for which the sum of the distances to n given points is minimum},
An. Oper, Res. 167 (2009),7–41.
\end{thebibliography}
\end{document}